\documentclass{amsart}

\usepackage{amsmath,amssymb,amsfonts,graphics,color}
\usepackage{epsfig}
\usepackage{latexsym}
\usepackage{euscript}
\usepackage{hyperref}
\newcommand{\ignore}[1]{}

\title{Projections in $L^1(G)$; the unimodular case}

\author[M. Alaghmandan]{Mahmood Alaghmandan}
\address{Department of Mathematical Sciences,
Chalmers University of Technology and the University of Gothenburg, S-412 96 G\"oteborg, Sweden}
\email{mahala@chalmers.se}
\author[M. Ghandehari]{Mahya Ghandehari}
\address{Department of Mathematical Sciences, University of Delaware, Newark, DE 19716, USA}
\email{mahya@udel.edu}
\author[N. Spronk]{Nico Spronk}
\address{Pure Mathematics, University of Waterloo,
 Waterloo, Ontario, Canada N2L 3G1}
 \email{nspronk@uwaterloo.ca}
\author[K. F. Taylor]{Keith F. Taylor}
\address{Department of Mathematics and Statistics, Dalhousie University, Halifax, Nova Scotia, Canada, B3H 3J5}
\email{keith.taylor@dal.ca}
\thanks{The first two authors were supported by  Fields institute postdoctoral fellowships and  postdoctoral fellowships at University of Waterloo. The second author was supported by an AARMS postdoctoral fellowship while at Dalhousie University. The final two authors were supported by NSERC of Canada Discovery Grants.}

\usepackage{xcolor}
\definecolor{blue1}{RGB}{32,78,170}
\definecolor{blue2}{RGB}{93,92,160}
\definecolor{blue3}{RGB}{40,51,202}
\definecolor{blue4}{RGB}{0,0,0}
\definecolor{purple1}{RGB}{128,0,128}

\date{\today}

\newtheorem{theorem}{Theorem}[section]
\newtheorem{corollary}[theorem]{Corollary}

\newtheorem{lemma}[theorem]{Lemma}
\newtheorem{proposition}[theorem]{Proposition}

\newcommand{\cal}{\mathcal}
\newcommand{\cA}{\mathcal{A}}

\newcommand{\cJ}{\mathcal{J}}
\newcommand{\gU}{{\operatorname{U}}}
\newcommand{\norm}[1]{\| #1\|}
\newcommand{\wG}{\widehat{G}}

\newcommand{\Hpi}{\mathcal{H}_{\pi}}

\newcommand{\cH}{\mathcal{H}}
\newcommand{\cT}{\mathcal{T}}
\newcommand{\sig}{\sigma}
\newcommand{\lam}{\lambda}
\newcommand{\fH}{{\cal H}}

\theoremstyle{definition}
\newtheorem{example}[theorem]{Example}
\def\SL {{\rm SL}_2({\mathbb R})}

\def\supp {{\rm supp}}

\def\H{{\cal H}}

\def\supp {{\rm supp}}

\begin{document}
\maketitle

\begin{abstract}
We consider the issue of describing all self-adjoint idempotents (projections) in $L^1(G)$ when $G$ is a unimodular locally compact group. The approach is to take advantage of known facts concerning subspaces of the Fourier-Stieltjes and Fourier algebras of $G$ and the
topology of the dual space of $G$. We obtain an
explicit description of any projection in $L^1(G)$
which happens to also lie in the coefficient space of
a finite direct sum of irreducible representations.
This leads to a complete description of all 
projections in $L^1(G)$ for $G$ belonging to a 
class of groups that includes $\SL$ and all
almost connected nilpotent locally compact groups.
\end{abstract}


\section{Introduction}
Let $G$ be a unimodular locally compact group, and $L^1(G)$ denote the Banach   $*$-algebra
of integrable functions on $G$.  Let $M(G)$ denote the Banach
$*$-algebra  of bounded regular Borel measures on $G$, and recall that the measure algebra $M(G)$ contains $L^1(G)$ as a closed ideal.
Self-adjoint idempotents in $L^1(G)$ (respectively $M(G)$) are called $L^1$-projections (respectively projection measures).
The study of projections originated with
Rudin \cite{Rud} and Helson \cite{hel}.   A full characterization of 
idempotents of the measure algebra of a  locally compact abelian group was 
obtained in \cite{Coh} through identifying such measures with certain subsets of the dual group. Note that in the abelian case, idempotents of $M(G)$
are automatically projection measures. 
For nonabelian compact groups, 
the orthogonality relations for coefficient
functions of irreducible representations
show that any pure positive definite 
function, properly scaled, is an $L^1$-projection. 
Such a projection is strongly minimal in the
sense defined later. Conversely, every strongly minimal $L^1$-projection is just a pure positive definite coefficient function, and is associated 
uniquely (up to equivalence) with a particular irreducible
representation (namely the unique representation in its ``{support}''). Moreover, every $L^1$-projection of a compact group is just a finite sum of such strongly minimal projections. 

The support of an $L^1$-projection $p$ of a locally compact group is  the collection of (equivalence classes of) all irreducible unitary representations $\pi$ which satisfy 
$\pi(p)\neq 0$. It turns out that an  $L^1$-projection in a compact or an abelian locally compact group can be understood using its support. 
For a noncompact 
unimodular group $G$, it was  shown
in \cite{bar} and independently in \cite{val} that, similar to the case of a compact group,  strongly minimal $L^1$-projections are singly supported. That is,
for a strongly minimal projection $p\in L^1(G)$,
 there is a unique (up to equivalence
of representations) irreducible representation $\pi$ of $G$ such that $\pi(p)\neq 0$. 
Moreover, the representation $\pi$ is an 
open point in the dual space of $G$, and $p$ is nothing but a positive definite coefficient function of $\overline{\pi}$.
%Conversely, if $\pi$ is an open point
%in the dual space of $G$ for certain locally compact groups, there are nonzero
%coefficient functions of $\pi$ which are
%strongly minimal projections in $L^1(G)$.
 
To our knowledge, the first explicit construction of projections in $L^1(G)$ for
a nonunimodular $G$ was carried out by
Eymard and Terp \cite{EyT} for the
group of affine transformations of 
any locally compact field. 
In \cite{GrKT} and \cite{KaT}, groups of the
form $G=A\rtimes H$, with $A$ abelian, were studied. The nature of the action of $H$ on $\widehat{A}$ determines whether or not 
there are nonzero projections in $L^1(G)$.
Groups of the form $A\rtimes H$ are often
nonunimodular and, although we know how to construct many examples of projections, we
are a long way from a characterization of
projections for nonunimodular groups that is
comparable to that for compact or abelian groups.

In this paper, we restrict our attention to 
unimodular groups $G$ with the purpose of
building on the results of \cite{bar} and \cite{val} and moving closer to a complete description of all
projections in $L^1(G)$. 
In particular, we study projections with finite support in detail and show that, for many groups (precisely the unimodular second-countable type I groups), the finite support of a projection identifies the smallest coefficient function space which contains the projection. This article provides partial generalizations to some earlier results about projections of certain unimodular groups. 
For $G$ a connected nilpotent group, all projections in $L^1(G)$
are explicitly described in \cite{KaT3}.
Some more 
headway was made in \cite{KaT2} for [FC$^-$]-groups; that is, groups for which every conjugacy class is relatively compact. Note
that nilpotent and [FC$^-$]-groups are
unimodular.

This article is organized as follows. We collate the necessary background and tools in Section \ref{section:notation}. In Section \ref{s:main-results} we prove that every $L^1$-projection can be represented by an element of the Fourier algebra. We then study projections that lie in certain subspaces of the Fourier algebra, namely coefficient spaces associated with  finite sums of irreducible representations. In Theorem \ref{t:projections-coefficent}, we show that every such projection is of a rather special form, { i.e.} it is just a finite sum of coefficient functions, where each summand is a strongly minimal $L^1$-projection in its own right. This in particular proves that every one of the irreducible representations has to be ``integrable''. We use the results of \cite{bar}, together with careful study of coefficient function spaces of irreducible representations, to prove this theorem. 
(It is worth mentioning that we know of no direct way to answer even a very simple version of this question, namely when the projection is assumed to be just the sum of coefficient functions of two inequivalent irreducible representations.)
In Section \ref{s:support-of-projections}, we study projections through their support, and show that in special cases, the support of the projection identifies its location in the Fourier algebra. 

As perhaps the 
most useful consequence of this study, Corollary 
\ref{high_point} provides a complete description of all 
projections in $L^1(G)$ when $G$ is a unimodular,
second countable, type I locally compact group
with the property that every compact open subset of
the dual of $G$ is a finite subset of the reduced dual.
This class of groups includes $\SL$ and any almost
connected nilpotent group.

We finally finish the paper with an application of projections to $*$-homomorphisms between $L^1$-algebras of (unimodular) locally compact groups. 

%%%%%%%%%%%%%%%%%%%%%%%%%%%%%%%%%%%%%%%%%%%%%%%%%%%%%%%%%%%%%%%%%%%%%%%%%%%%%

\section{Notation and background}\label{section:notation}

For the rest of this paper, $G$ is 
a unimodular locally compact group unless otherwise stated. Let Haar integration on $G$ be denoted by
$\int_G\cdots dx$.
For every function $f$ on $G$ and $y\in G$, 
define a new function $L_yf$ on $G$ such that $L_yf(x)=f(y^{-1}x)$.  The {\it convolution} of two Borel measurable functions $f$ and $g$,  written as $f*g$,   is defined by 
\[
f*g(x)=\int_G f(y) L_yg(x)\; dy
\]
if it exists. Also, the {\it involution} of a Borel measurable function $f$ is defined by $f^*(x):=\overline{f(x^{-1})}$ for every $x\in G$. (Note: If $G$ were nonunimodular, the involution would involve the modular function of $G$.)
With respect to the convolution and involution defined above, $L^1(G)$ forms a Banach $*$-algebra. 

By a {\it   representation} $\pi$ of $G$, we always mean a  homomorphism $\pi: G \rightarrow \mathcal{U}(\Hpi)$ which is SOT-continuous.
%;that is, continuous when  the group of unitary operators of the Hilbert space $\Hpi$, denoted by $\mathcal{U}(\Hpi)$, is equipped with the strong operator topology. 
A representation $\pi$ is called {\it irreducible} if 
$\Hpi$ has  no proper non-trivial closed $\pi$-invariant subspaces. Two representations  are called {\it equivalent} if
they are unitarily equivalent. If $\pi$ is a representation of $G$, we let $\pi$ denote the integrated (in
the weak operator sense) representation to $L^1(G)$.
\ignore{, that satisfies,
for all $\xi,\eta\in\Hpi$,
\[
\langle\pi(f)\xi,\eta\rangle =
\int_Gf(x)\langle\pi(x)\xi,\eta\rangle \,dx \ \ \ \ (f\in L^1(G)).
\]}

 The space of
equivalence classes of irreducible 
representations of $G$ is denoted by $\wG$.
There is a natural topology on $\wG$ that is
not, in general, Hausdorff (see \cite{induced}). 
The support of a representation $\pi$, denoted by $\supp(\pi)$, is the set of all representations in $\widehat{G}$ which are weakly contained in $\pi$.   For the left regular representation $\lambda$ of $G$, the support of $\lambda$ is called the \emph{reduced dual} of $G$ and is denoted by $\widehat{G}_r$.   For a  detailed account of representation theory of locally compact groups see  \cite{dixmier}.

For a representation $\pi$  on $G$ and $\xi,\eta\in \Hpi$, we define the corresponding
 {\it coefficient function} to be the function $\pi_{\xi,\eta}(x):=\langle\pi(x) \xi,\eta\rangle$ , for $x\in G$.
 Let $F_\pi$ denote the linear span of
 $\{\pi_{\xi,\eta}:\xi,\eta\in\Hpi\}$. Then 
 $F_\pi$ is a subspace of the space of bounded
 continuous complex-valued functions on $G$.
 
An irreducible  representation $\pi$ is called {\it integrable} if there  exists $\xi\in \Hpi$, $\xi\neq 0$, such that $\pi_{\xi,\xi}\in L^1(G)$. 
This is equivalent to the existence of  a dense subspace $\cH'$  of $\Hpi$ such that for all $\xi,\eta\in \cH'$, the coefficient function $\pi_{\xi,\eta}$ belongs to $L^1(G)$. 
An irreducible representation $\pi$ of $G$ is said to be {\it square-integrable} if there exist  non-zero vectors  $\xi,\eta \in \Hpi$ such that  $\pi_{\xi,\eta}\in L^2(G)$. Note that every integrable representation is square-integrable but the converse is not true. When $\pi$ is a square-integrable representation of a unimodular group, every coefficient function  of $\pi$ is square-integrable. 
See
Chapter 14 of  \cite{dixmier} for the basic properties of square-integrable and integrable
representations of unimodular groups. 

Square-integrable representations satisfy   orthogonality relations similar to the ones held 
for coefficient functions of   irreducible
representations of compact groups.
{In particular,  let $\sigma=\bigoplus_{i=1}^n \pi_i$ for  mutually non-equivalent square-integrable representations $(\pi_i)_{i=1}^n$, and  $\xi_i,\xi_i', \eta_i,\eta_i' \in {\cal H}_{\pi_i}$ for $i\in 1,\ldots, n$. Then for $\xi=\bigoplus_{i=1}^n \xi_i$, $\xi'=\bigoplus_{i=1}^n \xi'_i$, $\eta=\bigoplus_{i=1}^n \eta_i$, and $\eta'=\bigoplus_{i=1}^n \eta'_i$ we have
\begin{eqnarray}
\int_{G} \langle\xi,\sigma(x)\eta\rangle \overline{\langle\xi',\sigma(x)\eta'\rangle} dx &=& \sum_{i,j=1}^n  \int_{G} \langle\xi_i,\pi_i(x)\eta_i\rangle \overline{\langle\xi_j',\pi_j(x)\eta_j'\rangle} dx \label{eq:orthogonlity-relation-unimodular}\\
&=& \sum_{i=1}^n \frac{1}{k_i}\langle\xi_i,\xi_i'\rangle \langle \eta'_i,\eta_i\rangle\nonumber
\end{eqnarray}
where each positive real quantity $k_i$ is called the \emph{formal dimension} of $\pi_i$.  For such a
representation $\sigma$, $F_\sigma$ has additional structure. With $\xi,\eta,\xi',\eta'$ as above,
\begin{equation}\label{eq:convolution-of0coefficient}
\sigma_{\xi,\eta}* \sigma_{\xi',\eta'} = \sum_{i=1}^n \frac{1}{k_i} \langle \xi_i,\eta'_i\rangle \pi_{i,\xi'_i,\eta_i}  \in F_\sigma,
\end{equation}
where $\pi_{i,\xi_i,\eta_i} =\langle \pi_i(\cdot)\xi_i,\eta_i\rangle$.
Using  (\ref{eq:convolution-of0coefficient}), we observe that if $G$ is unimodular and $\sigma$ is a  {direct sum of finitely many   square-integrable representations},   then  $F_\sigma$ forms a $*$-algebra, where $*$ is the involution.

 For a locally compact group $G$, let $B(G)$ denote the set of all coefficient functions generated by  representations of $G$. Eymard first introduced $B(G)$  for a general locally compact group in \cite{eymard}.  Clearly, $B(G)$ is  an algebra with respect to the pointwise operations. Eymard showed  that $B(G)$ is in fact a Banach algebra with the norm defined as follows. For each $u\in B(G)$, $\norm{u}_{B(G)}:=\inf \norm{\xi}\norm{\eta}$ where the infimum is taken over all possibilities of representations $\sigma$ of $G$ and $\xi,\eta\in \H_\sigma$ with $u(x)=\langle\sigma(x)\xi,\eta\rangle$. The Banach algebra $B(G)$ is called the {\it Fourier-Stieltjes algebra} of the group $G$. Further,  $B(G)$ is invariant with respect to left and right translations by elements of $G$.

For a representation $\sigma$ of $G$, $F_\sigma$ is a subspace of $B(G)$ which is not necessarily closed. The closure of $F_\sigma$ with respect to the norm of $B(G)$ is denoted by $A_\sigma(G)$, or $A_\sigma$ when there is no risk of confusion. 
These subspaces were defined and studied by Arsac in \cite{arsac} where it was shown that
$$A_\sigma=\left\{\sum_{j=1}^\infty\sigma_{\xi_j,\eta_j}:\xi_j ,\eta_j  \in  {\cal H}_{\sigma},  \sum_{j=1}^\infty \|\xi_j\|\|\eta_i\|<\infty\right\}.$$
In addition, every $u$ in $A_\sigma$ can be represented as $u=\sum_{i=1}^\infty \sigma_{\xi_i,\eta_i}$ in such a way that 
\[
\|u\|_{B(G)}= \sum_{j=1}^\infty \|\xi_j \|\|\eta_j \|.
\]
 Moreover, the subspace $A_\sigma$ can be realized as a quotient of the trace class operator algebra of ${\cal H}_{\sigma}$, ${\cal T}(\cH_\sigma)$,  through 
the map $\psi$  defined as
\[
\psi:  {\cal T}( {\cal H}_{\sigma})   \longrightarrow  A_\sigma, \; \;
  \psi(T)(x) =\operatorname{Tr}(T \sigma(x)), 
\]
for every $T \in  {\cal T}( {\cal H}_{\sigma})$ and  $x\in G$.
In the special case where $\sigma$ is irreducible, the above map defines an isometry. In particular, we conclude that $\| \sigma _{\xi ,\eta }\|_{B(G)}= \|\xi \|\|\eta \|$.} 
In our computations, we will use the following proposition which is merely a weaker version of Corollaire (3.13) of \cite{arsac}.
\begin{proposition}\label{p:Arsac-direct-sum}
Let $\sigma=\bigoplus_{i\in I}\pi_i$, where $\{\pi_i:\, i\in I\}$ is a collection of non-equivalent irreducible representations of $G$. Then $A_\sigma=\ell^1\text{-}\bigoplus_{i\in I} A_{\pi_i}$.
\end{proposition}

%\ap{If $\sig=\bigoplus_{i=1}^n \pi_i$ for  irreducible unitary representations $\pi_i$ of a locally compact group $G$.    Then for each $\xi=\bigoplus_{i=1}^n \xi_i, \eta=\bigoplus_{i=1}^n \eta_i \in \bigoplus_{i=1}^n \cH_{\pi_i}$,
%\begin{equation}\label{eq:direct-product-norm-Arsac}
%\norm{\sigma_{\xi,\eta}}_{B(G)} = \sum_{i=1}^n \norm{\xi_i}\norm{\eta}_i.
%\end{equation}
%}

If $\lambda$ denotes  the left regular representation of $G$, $A_\lambda$ turns out to be a closed ideal of $B(G)$, and is simply  denoted by $A(G)$.  
The algebra $A(G)$, called the {\it Fourier algebra} of $G$,  was also introduced by  Eymard in \cite{eymard}. In particular,  Eymard  proved that each element of $A(G)$ can be written in the form of a coefficient function of $\lambda$, that is  $\lambda_{f,g}$ for some $f,g \in L^2(G)$.
 
An element $p$ in $L^1(G)$ is called a 
{\it projection} if $p*p=p=p^*$; that is, if
$p$ is a self-adjoint idempotent in $L^1(G)$.
Let $\mathcal{P}L^1(G)$ denote the set of
projections in $L^1(G)$.

For $p\in \mathcal{P}L^1(G)$, define the \emph{support} of $p$ to be 
$S(p):=\{\pi\in\wG:\pi(p)\neq 0\}$. For any
$p\in \mathcal{P}L^1(G)$, $S(p)$ is a compact
open subset of $\wG$ (see 3.3.2 and 3.3.7 of \cite{dixmier}), but $S(p)$ is not necessarily closed (see \cite[Example~2]{KaT3}). Thus, if $\wG$ has no
nonempty compact open subsets, then 
$\mathcal{P}L^1(G)=\{0\}$.

The set $\mathcal{P}L^1(G)$  carries a partial order $\leq$ that is  $q\leq p$ if $q*p=q$
(or equivalently $p*q=q$) for  $p,q\in 
\mathcal{P}L^1(G)$.
A nonzero $p\in\mathcal{P}L^1(G)$ is called a {\it minimal projection} if for any other 
$q\in\mathcal{P}L^1(G)$, $q\leq p$  implies that $q$ is either $p$ or $0$. A projection $p$ in $L^1(G)$ is called {\it strongly minimal} if the left ideal $L^1(G)*p$ is a minimal left ideal in $L^1(G)$. Equivalently,  for a strongly minimal projection $p$, 
\begin{equation}\label{eq:alpha-map}
p*f*p=\alpha_f p
\end{equation} 
where $\alpha_f\in\mathbb{C}$, for every $f\in L^1(G)$. 
It is clear that  strong minimality implies  minimality of a  projection. But the following example shows that the  converse is not true.

\begin{example}\label{min-stmin}
\ignore{Let $G$ be a non-compact locally compact abelian group, admitting a compact open subgroup for 
which $K^\perp\cong\widehat{G/K}$ is a non-singleton connected group.  
Then for any $\chi\in\widehat{K}$, the measure $\chi m_K$ is a minimal but not strongly minimal 
idempotent.  Indeed $(\chi m_K)\ast L^1(G)=\{f\in L^1(G):f(sk)=\chi(k)f(s)\text{ for all }
s\text{ in }G,\,k\text{ in }K\}$.  This can be routinely checked to have image, under the Fourier 
transform, $1_{\chi K^\perp}A(\widehat{G})\cong A(\widehat{G/K})$.  The connectivity assumption
of $\widehat{G/K}$ shows that this algebra admits no proper projections, so there are no projections
properly dominated by $\chi m_K$ in $L^1(G)$.
For instance, let $G=\Bbb{Z}_2 \times \Bbb{Z}$ for $K=\Bbb{Z}_2$, the cyclic group of  order two. Therefore, $\wG=\Bbb{Z}_2 \times \Bbb{T}$. Based on the characterization of projections, $L^1(G)$ has two minimal projections $\frac{1}{2} (\delta_{(0,0)} \pm \delta_{(1,0)})$ which are corresponding to connected  compact open subsets $\{0\}\times \Bbb{T}$ and $\{1\}\times \Bbb{T}$. But neither of these two projections are  strongly minimal.}
Let $G$ be a noncompact abelian group whose dual $\widehat{G}$ admits a connected open compact subgroup $K$, whence $G$ itself admits a compact open subgroup $K_\perp$ which has at least two distinct elements. Then by standard theory of the Fourier transform, the normalized Haar measure $m_{K_\perp}$ admits transform the indicator function $1_K$, and is clearly minimal, but not strongly minimal.
\end{example}

Strongly minimal projections have been studied in \cite{bar} and \cite{val} where they were called ``minimal". Since in this article we study two types of minimality for projections,  namely minimal and strongly minimal, we use different terminology. 
It has been shown that, for unimodular groups,  there is a one-to-one correspondence between the set of equivalence classes of integrable representations  and   strongly minimal projections. 
We will clarify 
the relation between a strongly minimal
projection and the corresponding 
irreducible representation, for
unimodular groups, in Section~\ref{s:main-results}.

% The following is a report of the main results that we need for strongly minimal projections. 

%\begin{proposition}\label{p:strongly-minimal-projections-are-coefficient}
%Let $G$ be a unimodular locally compact group and $p$ a strongly minimal projection in $L^1(G)$. If $\sigma$ is a unitary representation  of $G$ such that $\sigma(p)\neq$, then  $p$ is a coefficient function of $\overline{\pi}$ for  an irreducible subrepresentation $\pi$ of $\sigma$.
%\end{proposition}

%\begin{proof}
%By \cite[Proposition~1]{bar}, the representation  $\pi_p \in \widehat{G}$ induced by strongly minimal projection $p$  is a subrepresentation of $\sigma$. 
%Since for each $f\in L^1(G)$, $p*f*p=\alpha_f p$ for some constant $\alpha_f\in \Bbb{C}$. 
%Following the construction of  \cite{bar},  $\pi_p(f)(g*p) = f*g*p$ for all $f, g\in L^1(G)$.
% Hence,
%\[
%\overline{p}(x)=p*\delta_x*p(e)=\langle \delta_x*p,p\rangle p(e)=\langle \pi_p(x)p,p\rangle p(e)\ \ \ (x\in G).
%\]
%Therefore, $p$ is a coefficient function with respect to the representation $\opi_p$. 
%\end{proof}

%%%%%%%%%%%%%%%%%%%%%%%%%%%%%%%%%%%%%%%%%%%%%%%%%%%%%%%%%%%%

\section{Main results}\label{s:main-results}
Our objective for this paper is to study $L^1$-projections of  unimodular groups.  Our motivation is the result of Barnes in \cite{bar} which states that every strongly minimal $L^1$-projection of a unimodular group is a coefficient function of an  integrable representation.    In particular, strongly minimal projections of $L^1(G)$ lie
 in some $A_\pi$
with $\pi$ integrable (which implies 
%$\pi$ is square-integrable, which implies 
$A_\pi\subseteq A(G)$). 

We begin this section with an analogue  key observation on idempotents in $L^1(G)$. Note that   unimodularity of $G$  guarantees that many of the significant dense left ideals of $L^1(G)$ are (two-sided) ideals. 
Let us recall that an element $p$ of an algebra $\cA$ is called an {idempotent} if  $p^2=p$.
% 
%The following lemma is needed to show that all idempotents in $L^1(G)$ belong to every dense ideal of it.
%\begin{lemma}\label{l:characters-in-ZS^1(G)}
%Let ${\cal A}$ be a Banach algebra and ${\cal J}$ be a dense  ideal of ${\cal A}$. Then for each idempotent element $p \in {\cal A}$, $p$ belongs to ${\cal J}$.
%\end{lemma}
%
%\begin{proof} Suppose $p$ is a nonzero 
%idempotent in ${\cal A}$.
% Since $\cJ$ is dense in $\cA$, there is some $u\in \cJ$ such that $\norm{u-p}<\frac{1}{\|p\|}$. Define $b:=\sum_{n=1}^\infty (p-up)^n \in \cA$.  Note that $bp=b$. Moreover, $b(p-up)=b - (p-up)$. Therefore, $up+(b-b(p-up))=p$. On the other hand, $up+(b-b(p-up))=up+ bup \in \cJ$. This implies that $p\in \cJ$.
%\end{proof}
%
%Taking ${\cal A}=L^1(G)$ and ${\cal J}=L^1(G)\cap A(G)$ or $L^1(G)\cap L^p(G)$ for any $1\leq  p \leq \infty$, one has the following corollary. 
%
The  following proposition is formerly proved in  \cite[Theorem~8]{kotz} in a more general setting. We present the  proof here to be self-contained. 

\vskip1.0em

\begin{proposition}\label{p:every-idempotent}
Let $G$ be a unimodular locally compact group and $p$ be an idempotent in $L^1(G)$. Then, $p\in A(G)\cap L^r(G)$ for  every $1<r<\infty$.
\end{proposition}

\begin{proof}
Suppose $p$ is a non-zero  idempotent in $L^1(G)$. Let $\cJ$ be any one of the  ideals     $A(G) \cap L^1(G)$ or $L^1(G) \cap L^r(G)$ for some $1\leq r <\infty$.
Since $\cJ$ is dense in $L^1(G)$, there is some $u\in \cJ$ such that $\norm{u-p}_1< \|p\|_1^{-1}$. Define $b:=\sum_{n=1}^\infty (p-u*p)^{*n} \in L^1(G)$, where ${*n}$ denotes the $n$-fold convolution.
Note that $b*p=b$. Moreover, $b*(p-u*p)=b - (p-u*p)$. Therefore, $u*p+(b-b*(p-u*p))=p$. On the other hand, $u*p+(b-b*(p-u*p))=u*p+ b*u*p \in \cJ$. This implies that $p\in \cJ$.
%We only need to recall that $A(G)\cap L^1(G)$ and $L^r(G) \cap L^1(G)$ form dense ideals in the group algebra of $G$. 
\end{proof}
%\remark\label{r:two-sided-ideal}

Note that the proof of Proposition~\ref{p:every-idempotent}   does not work  for nonunimodular locally compact groups, as having two-sided ideals is  essential for the proof. 
%For a general locally compact group, a similar argument shows that for every idempotent $p\in  L^1(G)$, there exists some $u\in A(G)$ such that $p=p* \check{\overline{u}}\Delta^{-1}=u*p$,
%where $\Delta$ is the modular function of $G$. 

The following gives most of \cite[Theorem 1]{bar}, but from a perspective more suitable to our needs.

\begin{proposition}\label{p:stongly-minimal-integrable}
Let $G$ be a unimodular locally compact group and $p$ a projection in $L^1(G)$.
\begin{itemize}
\item[(i)]{Let $\pi=\lambda(\cdot)|_{L^2(G)\ast\overline{p}}$.  Then $p\in A_{\pi}$.}
\item[(ii)]{If $p$ is strongly minimal, then $\pi$ is irreducible and integrable. Further, $p=\pi_{\overline{p},\overline{p}}$.}
\end{itemize}
\end{proposition}

\begin{proof}
 Since $G$ is unimodular, and $p^*=p$, we have that $\check{p}$ is equal to $\overline{p}$ a.e. where $\check{p}(s):=p(s^{-1})$  and $\overline{p}(s):=\overline{p(s)}$.  Proposition~\ref{p:every-idempotent} tells us that $p\in A(G)\cap L^2(G)$.  Hence
we have
\[
p=p\ast p=p\ast\check{\overline{p}}=\langle p,\lambda(\cdot)p\rangle=\langle \lambda(\cdot)\overline{p},\overline{p}\rangle,
\]
which gives (i).

Let $p$ be a strongly minimal projection. Let $u$ be the element   in $L^\infty(G)$ which is associated with the linear functional $f \mapsto \alpha_f$  defined in  (\ref{eq:alpha-map}), { i.e.}  
\begin{equation}\label{eq:defining-u}
p\ast f\ast p=(\int_G uf)\,p\ \ \ \ \  (\text{for $f\in L^1(G)$}).
\end{equation}
  Notice that $\int_G u(p\ast f)=\int_G (\check{p}\ast u)f=\int_G (\overline{p}\ast u)f$. 
So for every $f\in L^1(G)$, 
\[
(\int_G  uf)\, p = p*f*p= p*(p*f)*p= (\int_G  u (p*f)) \, p = (\int_G (\overline{p}\ast u)f) \,p\]
which implies that   $\overline{p}\ast u=u$.  Likewise $u\ast\overline{p}=u$.  Now if $(u_\iota)$ is a net from $C_c(G)$
which is weak$^*$ convergent  to $u$, then we have
\[
u=\overline{p}\ast u\ast\overline{p}=\text{w*-}\lim_\iota \overline{p}\ast u_\iota \ast\overline{p}
=\lim_\iota\left(\int_G u_\iota\overline{u}\right)\overline{p}.
\]
In particular $\alpha:=\lim_\iota\left(\int_G u_\iota\overline{u}\right)$ exists and $u=\alpha\overline{p}$.
But $p=p\ast p\ast p=\alpha(\int_G p\overline{p})p$, so $\alpha=\|p\|_2^{-2}$. 
% (since a projection on a Hilbert space is norm one).

Now we follow a procedure in [2].  If $\sigma$ is any representation for which
$\sigma(\overline{p})\not=0$, find  $\xi$ in $\cH_\sigma$ such that $\sigma(\overline{p}) \xi=\xi$ and $\norm{\xi}^2=\alpha^{-1}$.  Interchanging roles of $p$ and $\overline{p}$ in (\ref{eq:defining-u}), we have for $f$ in $L^1(G)$ that
\[
\| \sigma(f)\xi\|^2=\langle \sigma(\overline{p}\ast f^*\ast f\ast \overline{p})\xi,\xi\rangle
=\alpha\norm{\xi}^2\int_G (f^*\ast f)\,p=\langle f^*\ast f\ast \overline{p},\overline{p}\rangle
=\|f\ast\overline{p}\|_2^2
\]
where the fact that $p=\pi_{\overline{p},\overline{p}}$ was used in the penultimate equality.
Hence $U:L^1(G)\ast\overline{p}\to \cH_\sigma$ given by $U(f\ast\overline{p})=\sigma(f)\xi$ extends
to an isometry from $L^2(G)\ast\overline{p}$ to $\cH_\sigma$ which intertwines $\pi$ and $\sigma$.
In particular with choice of irreducible $\sigma$, we see that $\pi$ is necessarily irreducible as well.
Since $p=\pi_{\overline{p},\overline{p}}$ is integrable, $\pi$ is an integrable representation.
%,\cite[14.5.1]{dixmier} shows that $\pi$
\end{proof}

The following remark gives the converse to (ii), above.

\begin{remark}\label{p:strongly-minimal}
Let $G$ be a unimodular locally compact group and $\pi$ be
an integrable irreducible representation of $G$.  Then there is a dense subspace
$\cH_\pi'$ of $\cH_\pi$ consisting of elements $\xi$ for which $\pi_{\xi,\xi}$ is a multiple
of a projection $p_\xi$ in $L^1(G)$.  Furthermore, by the calculation of \cite[Lemma 2.2]{val}, each $p_\xi$ is strongly minimal.
\end{remark}

\medskip

\ignore{
\begin{proof}
  The dense subspace $\cH_\pi'$ is exactly the one mentioned in the introduction.
For each $\xi$ in $\cH_\pi'$, ${\pi_{\xi,\xi}}^*=\pi_{\xi,\xi}$ in $L^1(G)$, and the orthogonality 
relation  (\ref{eq:convolution-of0coefficient})   implies that $\pi_{\xi,\xi}\ast \pi_{\xi,\xi}=\frac{\|\xi\|^2}{k_\pi}\pi_{\xi,\xi}$.
So  $p_\xi=\frac{k_\pi}{\|\xi\|^2}\pi_{\xi,\xi}$ is a projection.  Moreover, the calculation of \cite[Lemma 2.2]{val} shows that 
$p_\xi$ is strongly minimal. 
\end{proof}
}

The preceding  propositions give a  description of strongly minimal projections. 
{In what follows, we will study projections in $L^1(G)$ given that they belong to certain subspaces of the Fourier algebra.} We begin with the following lemma which says that, similar to the compact case, $A_\pi$ for a square-integrable representation $\pi$, is a Banach $*$-algebra with respect to convolution.

\begin{lemma}\label{l:unimodualr}
Let $\sig=\bigoplus_{i=1}^n \pi_i$ for  square-integrable representations $\pi_i$ of a unimodular locally compact group $G$.   Then  $A_\sig \subseteq L^2(G)$ and $\sqrt{k_\sigma}\norm{\cdot}_2\leq \norm{\cdot}_{B(G)}$ on $A_\sig$ where $k_\sigma=\min\{k_{\pi_i}: i=1,\ldots,n\}$. Furthermore, $(A_\sigma, k_\sigma^{-1}\norm{\cdot}_{B(G)})$ is a Banach $*$-algebra when it is equipped with  convolution.
\end{lemma}

\begin{proof}
By orthogonality relations stated in   (\ref{eq:orthogonlity-relation-unimodular}), for each $\xi=\bigoplus_{i=1}^n \xi_i, \eta=\bigoplus_{i=1}^n \eta_i \in \bigoplus_{i=1}^n \cH_{\pi_i}$, $\sigma_{\xi, \eta} \in L^2(G)$,  
\[
\norm{\sigma_{\xi,\eta}}_2^2= \sum_{i=1}^n\frac{1}{k_{\pi_i}} \norm{\xi_i}^2\norm{\eta_i}^2  \leq \frac{1}{k_\sigma} \left(\sum_{i=1}^n  \norm{\xi_i} \norm{\eta_i} \right)^2 \leq \frac{1}{k_\sigma} \norm{\sigma_{\xi,\eta}}_{B(G)}^2,
\]
where we used Proposition~\ref{p:Arsac-direct-sum} in the last inequality.
Let $u=\sum_{k=1}^\infty \sigma_{\xi_k, \eta_k} \in A_\sigma$ be represented such that $\norm{u}_{B(G)}=\sum_{k=1}^\infty \norm{\xi_k}\norm{\eta_k}$. 
Then, 
\[
\norm{u}_2 \leq \sum_{k=1}^\infty \norm{ \sigma_{\xi_k, \eta_k}}_2 \leq \sum_{k=1}^\infty \frac{1}{\sqrt{k_\sigma}}\norm{\sigma_{\xi_k,\eta_k}}_{B(G)}  \leq \frac{1}{\sqrt{k_\sigma }}\sum_{k=1}^\infty \norm{\xi_k}\norm{\eta_k}= \frac{1}{\sqrt{k_\sigma }}\norm{u}_{B(G)}.
\]
Therefore, $A_\sigma \subseteq L^2(G)$. Moreover, since $G$ is unimodular, for every $u,v\in A_\sigma$, their convolution is defined, and $
u*v  = \langle \lambda(\cdot) \check{v}, \overline{u}\rangle$. Thus, 
%\begin{equation}\label{eq:Api-is-normed-alegbra}
\[
\norm{u*v}_{B(G)} \leq \norm{\overline{u}}_2\norm{\check{v}}_2 \leq k^{-1}_\sigma \norm{u}_{B(G)} \norm{v}_{B(G)}.
\]
This completes the proof.
%\end{equation}
%Recall that $F_\sigma$, as a $*$-algebra,  is dense in $A_\sigma(G)$.
%Equation (\ref{eq:Api-is-normed-alegbra}) and an approximation argument   imply  that $A_\sigma(G)$ is a $*$-algebra.
\end{proof}

\medskip
The following is a partial generalization of \cite[Theorem 3]{KaT3}, where conditions on the set $S(p)$
were assumed.  We take the perspective of assuming $p$ itself consists of certain types of
matrix coefficients.
{
\begin{theorem}\label{t:projections-coefficent}
  Let $G$ be a unimodular locally compact group.  Let $\pi_1,\dots,\pi_n$
be a family of pairwise inequivalent members of $\widehat{G}$ and $\sigma=\bigoplus_{i=1}^n\pi_i$.
If $p$ in $L^1(G)\cap A_\sigma$ is a projection which belongs to no $A_{\sigma'}$ for any
proper subrepresentation $\sigma'$ of $\sigma$, then
\begin{itemize} 
\item[(i)]{ $p=\sum_{i=1}^n p_i$ where each $p_i$ is a projection in $L^1(G)\cap A_{\pi_i}$
and $p_i\ast p_{i'}=0$ for $i\not=i'$;}
\item[(ii)]{each $p_i=\sum_{j=1}^{r_i}p_{ij}$ where each $p_{ij}$ is strongly minimal and
$p_{ij}\ast p_{ij'}=0$ if $j\not=j'$;}
\item[(iii)]{each $\pi_i$ is integrable; and}
\item[(iv)] $S(p)=\{\overline{\pi_1},\ldots,\overline{\pi_n}\}$.
\end{itemize}
\end{theorem}
}

\begin{proof}
First note that $A_\sigma\cap A(G)\neq \{0\}$, since  $p\in L^1(G)\cap A(G)$  by  Proposition~\ref{p:every-idempotent}.  Thanks to \cite[(3.12)]{arsac}
there is a subrepresentation $\sigma'$ of $\sigma$ for which
$A_{\sigma'}=A_\sigma\cap A(G)$, but then $p\in A_{\sigma'}$, and our assumptions ensure that
$\sigma'=\sigma$.  In particular, for  each $i$,  $A_{\pi_i}\subseteq A_\sigma\subseteq A(G)$, and hence by 
\cite[14.3.1]{dixmier} each $\pi_i$ is square-integrable.  Thus by
Lemma~\ref{l:unimodualr}, $(A_\sig,k_\sigma^{-1}\|\cdot\|_B)$ is an involutive Banach
algebra when equipped with convolution. 

Recall that  $A_\sig^*\cong VN_\sig
=\ell^\infty\text{-}\bigoplus_{i=1}^n \mathcal{B}(\cH_{\pi_i})$.  We have  the usual duality
$\cT(\overline{\cH})^*\cong\mathcal{B}(\cH)$ given by $( \overline{\eta}\otimes\overline{\xi}^*, T)\mapsto
\mathrm{Tr}(\xi\otimes\eta^* T)=\langle T\xi,\eta\rangle$, 
where $\overline{\eta}\otimes\overline{\xi}^*$ is the rank-one operator on $\overline{\cH}$ 
given by $\overline{\zeta}\mapsto\langle\xi,\zeta\rangle\overline{\eta}$.
Combining these facts gives us an isometric Banach space isomorphism
\[
\phi:A_\sig\to\cT=\ell^1\text{-}\bigoplus_{i=1}^n \cT(\overline{\cH}_{\pi_i}),\quad
\sum_{i=1}^n\sum_{j=1}^\infty \pi_{i,\xi_{ij},\eta_{ij}}\mapsto
\sum_{i=1}^n \sum_{j=1}^\infty \overline{\eta}_{ij}\otimes\overline{\xi}_{ij}^*,
\]
where $ \pi_{i,\xi_{ij},\eta_{ij}}=\langle\pi_i(\cdot)\xi_{ij},\eta_{ij}\rangle$.
Consider the new mapping $\Phi$ as follows.
\[
\Phi: (A_\sig,k_\sigma^{-1}\|\cdot\|_{B(G)}) \to \cT, \quad
\sum_{i=1}^n\sum_{j=1}^\infty \pi_{i,\xi_{ij},\eta_{ij}}\mapsto
\sum_{i=1}^n \frac{1}{k_{\pi_i}}\sum_{j=1}^\infty \overline{\eta}_{ij}\otimes\overline{\xi}_{ij}^*,
\]
  where $k_\sigma$ is the constant from the preceding lemma.
It is straightforward to check that $\Phi$ is a continuous bijective algebra homomorphism when the domain is endowed
with the convolution product.  Indeed, one checks that $\Phi$ is a homomorphism
on the dense subspace $F_\sig$ of
finite sums of matrix coefficients of $\sig$, and observes that $\Phi(F_\sig)$ is
the space of finite rank operators, which is dense in $\cT$.  
%Furthermore, $\Phi$ is clearly bounded below.

Now we let $\mathcal{A}=p\ast L^1(G)\ast p$. Since $L^1(G)\cap A_\sig$ is a left ideal in
$L^1(G)$, $\mathcal{A}$ is an involutive convolution  algebra with unit $p$, and
 $\mathcal{A} \subseteq A_\sig$.  Hence $\Phi(\mathcal{A})$ is a unital involutive subalgebra
of $\cT$, whence of the algebra of compact operators $\mathcal{K}
=c_0\text{-}\bigoplus_{i=1}^n\mathcal{K}(\cH_{\pi_i})$.  So if $P=\Phi(p)$, $P$ is a compact idempotent and thus it  is finite dimensional.
Moreover, we
have that $\Phi(\mathcal{A})\subset P\mathcal{K}P$ which is a finite dimensional $*$-algebra,
hence semisimple and thus by Wedderburn's theorem
isomorphic to a finite direct sum of full matrix algebras $\bigoplus_{j=1}^N M_{r_j}(\mathbb{C})$.
Let $\{E^j_{kl}\}_{k,l=1}^{r_j}$ be the matrix units of  $M_{r_j}(\mathbb{C})$.
Note  that for each $j\in 1,\ldots, N$ and $k\in 1, \ldots, r_j$, $ E^{j}_{kk}$ is a strongly minimal projection in $\bigoplus_{i=1}^N M_{r_i}(\mathbb{C})$, as $E^{j}_{kk} (\bigoplus_{i=1}^N M_{r_i}(\mathbb{C})) E^{j}_{kk} = \Bbb{C} E^{j}_{kk}$. 
Hence,  $p_{jk}:=\Phi^{-1}(E^j_{kk})$ is a strongly minimal projection in ${\cal A}$ with $p_{jk}\leq p$, and subsequently for each $f\in L^1(G)$ we have
\[
p_{jk} * f * p_{jk} = p_{jk} * (p*f*p) * p_{jk} \in \mathcal{A}.
\]
We now appeal to  Proposition~\ref{p:stongly-minimal-integrable}, and observe that $p_{jk}$ must be a coefficient function of an integrable representation. On the other hand $p_{jk}\in A_\sigma$, and therefore $p_{jk}$ is a coefficient function of $\pi_{i}$ for some $i$. 
In particular, $\pi_i$ is integrable. 

In the remaining, we prove that $N=n$.  In what follows, we assume that $\pi_i$ belongs to $\{ \pi_1, \ldots, \pi_n\}$.  We show the following two facts:
\begin{itemize}
\item[(a)]{ If $k\neq k'$ and $p_{jk}\in A_{\pi_i}$ then $p_{jk'}\in A_{\pi_i}$.}
\item[(b)]{ If $j\neq j'$ then $p_{jk}$ and $p_{j'l}$ do not belong to the same $A_{\pi_i}$. }
\end{itemize}

To prove (a), fix $j$ and $k \neq k'$.  Towards a contradiction, assume that $p_{jk}$ and $p_{jk'}$ are coefficient functions of representations  $\pi_{i}$ and $\pi_{i'}$ respectively,  with $i\neq i'$.   Since $A_{\pi_i}\ast A_{\pi_{i'}}=\{0\}$, for every $f\in L^1(G)$, we have $p_{jk} * f*p_{jk'}=0$ as $ f*p_{jk'}$ still belongs to $A_{\pi_{i'}}$. But this is a contradiction, as for $f=\Phi^{-1}(E^j_{kk'})$ we get  $ p_{jk} * f *  p_{jk'} =\Phi^{-1} (E^j_{kk} E^j_{kk'} E^j_{k'k'})=f \neq 0$.

For (b), recall that for each $f\in L^1(G)$, $p_{jk} * f* p_{j'l} = p_{jk} * (p* f* p) * p_{j'l}=0$, since $\Phi$ is a homomorphism. Now towards a contradiction, suppose  that $p_{jk}=\pi_{i,\xi,\xi}$ and $p_{j'l}=\pi_{i,\eta,\eta}$ for some $\xi,\eta\in \cH_{\pi_i}$.    Since $\pi_i$ is irreducible, there is some $g\in L^1(G)$ such that $\langle \pi_i(\overline{g}) \eta, \xi\rangle \neq 0$. Therefore,
\[
p_{jk} * g * p_{j'l} = \pi_{i, \xi, \xi} * \pi_{i, \eta, \pi_i(\overline{g}) \eta} = \langle \xi, \pi_i(\overline{g}) \eta \rangle \pi_{i, \eta, \xi} \neq 0,
\]
which is a contradiction. 
So with $p_i:=\sum_{j=1}^{r_i} p_{ij}$, properties  (i) and (ii) hold. 

To prove (iv), note that  by Lemma 1.1 of \cite{val}, the support of a strongly minimal projection is a singleton. In fact, for a nonzero $L^1$-projection of the form $\pi_{\xi,\xi}$, we have  $S(\pi_{\xi,\xi})=\{\overline{\pi}\}$, since 
$\langle\overline{\pi}(\pi_{\xi,\xi})\overline{\xi},\overline{\xi}\rangle=\|\pi_{\xi,\xi}\|_2^2>0$. 
Thus for every $i$, $S(p_i)= \{\overline{\pi_i}\}$, since $p_i$ is a finite sum of strongly minimal projections, each of which is a coefficient function of $\pi_i$. This fact, together with the orthogonality relations for square-integrable representations, imply that $S(p)=\{\overline{\pi_1},\ldots,\overline{\pi_n}\}$. 
\end{proof}
%
%Theorem~\ref{t:projections-coefficent}, in particular, implies that the notions of minimality and strong minimality of projections in $A_\pi$ are equivalent.
%

\begin{corollary}
Let $\pi$ be an irreducible representation of a unimodular locally compact group $G$, and $p\in A_\pi$ be an $L^1$-projection. Then $p$ is a minimal projection if and only if it is strongly minimal.
\end{corollary}
\begin{proof}
By Theorem~\ref{t:projections-coefficent}, a projection $p$ in $A_\pi$ can be written as a finite sum of strongly minimal projections of the form $\pi_{\xi,\xi}$. Therefore, $p$ is minimal if and only if it is strongly minimal. 
\end{proof}

\section{Support of projections}\label{s:support-of-projections}
Recall that the support of an  $L^1$-projection $p$, denoted by $S(p)$, is the collection of all (equivalence classes of) irreducible representations $\pi$ of $G$ such that $\pi(p)\neq 0$. 
In this section, we show that the support sheds some light on the structure of the projection itself.
 This is evident in the abelian case, where the Fourier transform of a projection is just the characteristic function of the conjugate of its support. Recall that the support of a projection is always open and compact in the Fell topology of the dual.  For compact groups, the support of a projection is the finite set of irreducible representations which are used to construct the projection. 
We study similar cases (projections with finite support) for general unimodular groups in more detail.  
 
We start this section by a general observation linking the support of a projection and the support of its GNS representation. 

\begin{proposition}\label{p:second-countable}
 Let $G$ be unimodular, second countable and type I, and $p$ be a projection
in $L^1(G)$.  Then $S(p)\cap\widehat{G}_r$ is dense in $\supp\overline{\pi}_p$.
\end{proposition}

\begin{proof}
We have the following  Plancherel 
picture of the left and right regular representations (see \cite[Section~7.5]{folland}). There is a Borel subset
$B$ of $\widehat{G}$ which is dense in $\widehat{G}_r$ and a measure $\mu$ on $\widehat{G}$
which is carried by $B$ for which we have unitary equivalences
\[
\lam\cong\int^\oplus_{B}I\otimes \overline{\pi}\,d\mu(\pi)\text{ and }
\rho\cong\int^\oplus_{B}  {\pi} \otimes I \,d\mu(\pi)
\]
on
\[
L^2(G)\cong\int^\oplus_{B} {\fH}_\pi\otimes\overline{\fH}_\pi\,d\mu(\pi).
\]
 The reader may refer to \cite{arsac, folland, dixmier} for the theory of direct integrals of representations.  Note that the aforementioned presentation is slightly different  from (but equivalent to) the one in \cite[18.8.1]{dixmier}. 
%( This is related to the presentation given [Folland,(7.44)]).
Proposition~\ref{p:stongly-minimal-integrable} shows that the representation $\pi_p=\lam(\cdot)|_{L^2(G)\ast\overline{p}}$ on the Hilbert space $L^2(G)\ast\overline{p}$ with the cyclic vector $\overline{p}$
is the Gelfand-Naimark
cyclic representation of the positive-type element $p$.  Observe, then, that
\[
\rho(\overline{p})=\int^\oplus_{B}\pi(\overline{p})\otimes I \,d\mu(\pi),
\]
so
\begin{align}\label{eq:pipplancherel}
L^2(G)\ast\overline{p}&=\rho(\overline{p})L^2(G)
=\int^\oplus_{B}\pi(\overline{p})\fH_\pi\otimes\overline{\fH}_\pi\,d\mu(\pi) \notag \\
&=\int^\oplus_{\{\pi\in B:\pi(\overline{p})\not=0\}}\pi(\overline{p})\fH_\pi\otimes\overline{\fH}_\pi\,d\mu(\pi).
\end{align}
By  \cite[8.6.8 and 8.6.9]{dixmier}, it  follows that 
\[
\supp\pi_p={\rm cl}\{\pi\in B:\pi(\overline{p})\not=0\}
%={\rm cl}\{\pi\in B:\overline{\pi}(p)\not=0\},
\]
where we have used ${\rm cl} S$ to denote the closure of $S$, so as not to conflict with notation of conjugation.
It is clear that $S(\overline{p}) \cap \widehat{G}_r$ contains $\{ \pi \in B: \ \pi(\overline{p})\neq 0\}$, while also that $(S(\overline{p}) \cap \widehat{G}_r ) \cap ( \widehat{G}_r \setminus \operatorname{cl}\{ \pi \in B: \ \pi (\overline{p}) \neq 0\}) = \emptyset$. Interchanging  $p$ and $\overline{p}$, we obtain
the desired result.
\end{proof}

The following example shows that for a  totally disconnected algebraic group,   the support of an $L^1$-projection  does not necessarily lie in the reduced dual. However,  we know of no connected, unimodular, second countable and type I group $G$ for which the support of an $L^1$-projection does not lie in $\widehat{G}_r$.

\begin{example}\label{eg:algebraic-group}
Let $G=\operatorname{SL}_n(\mathbb{Q}_p)$ for $(n\geq 2)$. Then $G$ is type I, as it is a reductive $p$-adic group (see~\cite{bern}). Note that $G$ has an open compact subgroup $K=\operatorname{SL}_n(\mathbb{O}_p)$. Consider the projection $1_K$ in $L^1(G)$, and note that for the trivial character ${\bf 1}$ on ${G}$, ${\bf 1}(1_K)\neq 0$. But ${\bf 1}\notin \widehat{G}_r$, as $G$ is not amenable.
\end{example}
Proposition~\ref{p:second-countable}   tells us that 
\[
\pi_p=\int^\oplus_{B \cap S(\overline{p})} I\otimes  \overline{\pi} d\mu(\pi).
\] In a particular case, when $S(p)\cap \widehat{G}_r$ is finite, we can describe the projection as in the following theorem.

\begin{theorem}\label{t:single-supported-projection}
 Let $G$ be unimodular, second countable and type I, and $p$ be a projection
in $L^1(G)$. If $S(p)\cap\widehat{G}_r=\{\pi_1,\dots,\pi_n\}$, then $p\in A_\sig$ where 
$\sig=\bigoplus_{i=1}^n\overline{\pi_i}$, and $S(p)=\{\pi_1,\dots,\pi_n\}$.
\end{theorem}

\begin{proof}
Note that by (\ref{eq:pipplancherel}), the  measure representing $\pi_p$   is supported on $\{\overline{\pi}_1,\ldots, \overline{\pi}_n\}$; hence, the Plancherel measure $\mu$ admits each $\overline{\pi}_i$ as an atom.  
Then, letting $\sig=\bigoplus_{i=1}^n\overline{\pi}_i$, Proposition~\ref{p:stongly-minimal-integrable} shows that $p\in A_\sig$. 
It is easy to  see that for any proper subrepresentation $\sig'$ of $\sig$, $p\not\in A_{\sig'}$. This follows from the  orthogonality relations for square-integrable representations $\pi_i$ and the fact that  the support of $p$ contains $\{\overline{\pi}_1,\ldots, \overline{\pi}_n\}$. 
%We now   apply  Theorem~\ref{t:projections-coefficent} to finish the proof.
\end{proof}

\begin{corollary}\label{high_point}
Let $G$ be a unimodular, second countable, type I
locally compact group with the property that 
every compact open subset of $\widehat{G}$ is
a finite subset of $\widehat{G}_r$. Let $p$ be a projection in $L^1(G)$. Then, there exist mutually
inequivalent $\pi_1,\cdots,\pi_n\in\widehat{G}$ 
such that
\begin{itemize} 
\item[(i)]{ $p=\sum_{i=1}^n p_i$ where each $p_i$ is a projection in $L^1(G)\cap A_{\pi_i}$
and $p_i\ast p_{i'}=0$ for $i\not=i'$;}
\item[(ii)]{each $p_i=\sum_{j=1}^{r_i}p_{ij}$ where each $p_{ij}$ is strongly minimal and
$p_{ij}\ast p_{ij'}=0$ if $j\not=j'$;}
\item[(iii)]{each $\pi_i$ is integrable; and}
\item[(iv)] $S(p)=\{\overline{\pi_1},\ldots,\overline{\pi_n}\}$.
\end{itemize}
 
\end{corollary}
\begin{proof} Under these hypotheses, the compact open set $S(p)$ must be a finite subset of $\widehat{G}_r$. Now,
combine Theorem~\ref{t:projections-coefficent} with
Theorem \ref{t:single-supported-projection}.
\end{proof}

We remark that every compact open subset of $\widehat{G}$ is a finite subset of $\widehat{G}_r$
for $\mathrm{SL}_2(\mathbb{R})$, and for any almost connected nilpotent group (\cite[Theorem 4]{KaT3}).
No noncompact property (T) group enjoys this property; indeed consider the trivial representation
$\{{\bf 1}\}$.

\section{Application to homomorphisms of group algebras }
Let $p$ be a projection in $L^1(G)$. Following \cite{ross1}, define the set
%the set of all measures $\mu \in M(G)$ such that $\mu^**\mu=\mu*\mu^*=q$ and $\mu*q=\mu$ denoted by $\Gamma_q$.
\[
\mathbb{M}_p:=\{ \mu\in M(G): \; \mu^**\mu=\mu*\mu^*=p\; \text{and}\; p*\mu=\mu\}.
\]
We shall call this the \emph{intrinsic unitary group at $p$}. Note that since $L^1(G)$ is an ideal in $M(G)$, we see, in fact, that ${\mathbb M}_p\subseteq L^1(G)$.
One can equip $\mathbb{M}_p$ with the topology $\sigma(L^1(G),C_0(G))$ restricted to $\mathbb{M}_p$.  With convolution product, identity $p$, and inverses $f^{-1} := f^*$, $\mathbb{M}_p$ is a semi-topological group with continuous inversion. 

Let us make the assumption that {\it $G$ is a unimodular, second countable, type I group, for which
every compact open subset is a finite subset of $\widehat{G}_r$}.  Then by 
Corollary~\ref{high_point},
every $L^1$-projection $p$ admits the form
\begin{equation}\label{eq:form-of-projection}
p=\sum_{i=1}^n p_i\ \ \text{and }\  \ p_i=\sum_{j=1}^{r_i} k_{\pi_i} \pi_{i, \xi_j^{(i)},\xi_j^{(i)}}
\end{equation}
where $k_{\pi_i}>0$ is the formal dimension of $\pi_i$, and $\xi_1^{(i)},\ldots, \xi_{r_i}^{(i)}$ are unit vectors in $\cH_{\pi_i}$.
For notational convenience, we define $u\cdot p_i$, when $u$ is a unitary matrix of size $r_i$, to be $u\cdot p_i=\sum_{k,\ell=1}^{r_i} u_{k,\ell} k_{\pi_i} \pi_{i, \xi_k^{(i)},\xi_\ell^{(i)}}$

\begin{proposition}\label{p:unitary-group-of-SL2-type-groups}
With the assumptions given above, each intrinsic unitary group in $L^1(G)$ is of the form of 
\[
\mathbb{M}_p=\left\{ \sum_{i=1}^n u_i \cdot p_i:\; u_i\in \gU(r_i)\right\} \cong \prod_{i=1}^n \gU(r_i)
\]
when $p$ is a projection in $L^1(G)$ with $p=\sum_{i=1}^n p_i$  as in (\ref{eq:form-of-projection}), where $r_i\in\mathbb{N}$ and 
$U(r_i)$ is the group $r_i\times r_i$ unitary matrices,
for $1\leq i\leq n$.
. 
\end{proposition}

\begin{proof}
We saw in the proof of Theorem~\ref{t:projections-coefficent}, that a 
direct sum of matrix algebras $\bigoplus_{i=1}^nM_{r_i}(\mathbb{C})$
is $*$-isomorphic to $p\ast L^1(G)\ast p$, with the isomorphism given by
$(a_i)_{i=1}^n\mapsto \sum_{i=1}^n\sum_{k,\ell=1}^{r_i} a_{k,\ell} k_{\pi_i} \pi_{i, \xi_k^{(i)},\xi_\ell^{(i)}}$.
The structure of the unitary group follows immediately.
%Note that, for each $i\in 1, \ldots, n$, $\mathbb{M}_{p_i}$ is a group in  $p_iL^1(G)p_i$.  But as we saw in the proof of Theorem~\ref{t:projections-coefficent}, the homomorphism $\Phi$ maps $p_iL^1(G)p_i$ onto a $*$-subalgebra of $\cT (\overline{\cH}_{\pi_i})$ which is isomorphic to $M_{d_{\pi_i}}(\mathbb{C})$. Therefore the  local unitary group at  $p_i$ is of the form 
%\[
%\mathbb{M}_{p_i} = u\cdot p_i\ ,u\in \gU(d_i).
%\]
%The rest is straightforward.
\end{proof} 

The value of the above result lies in its application to the problem of constructing homomorphisms from $L^1(F)$  to $L^1(G)$, where $F$ is another locally compact group, as given in \cite{ross1}. 
Let $p$ be given as in (\ref{eq:form-of-projection}), and $\mathbb{M}_p$ as in 
Proposition~\ref{p:unitary-group-of-SL2-type-groups}.  Given any continuous homomorphism 
$\phi:F\to \prod_{i=1}^n\mathrm{U}(d_i)$, we can render
a $*$-homomorphism $\Phi_p:L^1(F)\to L^1(G)$ by
\begin{equation}\label{eq:strucPhip}
\Phi_p(f)=\sum_{i=1}^n\int_F f(s)\phi(s)_i\cdot p_i\,ds.
\end{equation}
We can use this to construct non-trivial homomorphisms form $L^1(F)$ to $L^1(G)$ where
there exists no non-trivial homomorphisms form $F$ to $G$.  For example we may let
$F={\rm{SU}}(n)$, and let $G={\rm SL}_2({{\mathbb R}})$, the reduced Heisenberg group $\mathbb{H}_r$,
or any finite group admitting an irreducible representation of dimension at least $n$.
Notice that if $F$ and $G$ are abelian and $G$ is compact then each $d_i=1$, and the $*$-homomorphism $\Phi_p$ corresponds to the piecewise
affine map $\widehat{G}\to\widehat{F}$ whose domain is $\{\overline{\pi}_1,\dots,\overline{\pi}_n\}$ and is given by $\overline{\pi}_i\mapsto\phi_i$.  In the case that $G$ is nonabelian and some $d_i>1$, then 
$\|p\|_1>1$ and $\Phi_p$ is necessarily non-contractive;  compare with \cite{greenleaf}.

Let us close with a modest characterization of homomorphisms described above.

\begin{proposition}\label{p:homo-regkernel}
Let $G$ satisfy the conditions of 
Corollary~\ref{high_point} and $F$ be any locally compact group.  A $*$-homomorphism
$\Phi:L^1(F)\to L^1(G)$ is of the form $\Phi=\Phi_p$, as in (\ref{eq:strucPhip}), if and only if
$\ker\Phi$ is a modular ideal of $L^1(F)$.
\end{proposition}

\begin{proof}
If $\Phi=\Phi_p$, as in (\ref{eq:strucPhip}), then $\Phi(L^1(F))\subseteq p\ast L^1(G)\ast p$.
As in the proof of Theorem~\ref{t:projections-coefficent}, we see that $\Phi(L^1(F))$ is
isomorphic to a $*$-subalgebra of a direct sum of full matrix algebras, and hence is unital,
whence $\ker\Phi$ is a modular ideal of $L^1(F)$.  Coversely, if $L^1(F)/\ker \Phi$
admits an identity, $q+\ker\Phi$, then $q^*+\ker \Phi$ is also the identity so
$p=\Phi(q)$ is a projection in $L^1(G)$.
Furthermore, $\Phi(L^1(F))\subseteq p\ast L^1(G)\ast p$.  Hence by the method of proof of
Theorem 3.8 of \cite{ross1}, $\Phi$ corresponds to a bounded homomorphism
$\Phi_M:M(F)\to p\ast M(G)\ast p=p\ast L^1(G)\ast p$ and hence to a continuous homomorphism
$\phi:F\to\mathbb{M}_p$, which, in turn, gives the form $\Phi=\Phi_p$, as in (\ref{eq:strucPhip}).
\end{proof}

\bibliographystyle{plain}
\bibliography{Projbib}

\end{document}